\documentclass[a4paper]{amsart}

\usepackage[english]{babel}

\usepackage{amssymb}
\usepackage{amsthm}
\usepackage{mathtools}
\usepackage{dsfont}
\usepackage{enumitem}
\usepackage{aliascnt}

\usepackage[T1]{fontenc}
\usepackage{lmodern}
\usepackage{microtype}

\usepackage[a4paper, inner=3cm,outer=3cm,top=3.3cm,bottom=3cm]{geometry}

\usepackage{xspace}

\usepackage[autostyle]{csquotes}

\usepackage{tikz}
\usetikzlibrary{matrix,arrows,patterns,decorations.pathreplacing}
\usepackage{pgfplots}

\usepackage[colorlinks, bookmarksdepth=3]{hyperref}

\makeatletter
\@ifpackageloaded{hyperref}{%
\DeclareRobustCommand{\SkipTocEntry}[5]{}}{%
\DeclareRobustCommand{\SkipTocEntry}[4]{}}
\makeatother

\theoremstyle{plain}
\newtheorem{thm}{Theorem}[section]
\newaliascnt{prop}{thm}
\newtheorem{prop}[prop]{Proposition}
\aliascntresetthe{prop}
\newaliascnt{cor}{thm}
\newtheorem{cor}[cor]{Corollary}
\aliascntresetthe{cor}
\newaliascnt{lem}{thm}
\newtheorem{lem}[lem]{Lemma}
\aliascntresetthe{lem}
\newaliascnt{conj}{thm}

\aliascntresetthe{conj}

\theoremstyle{definition}
\newtheorem{defn}[thm]{Definition}

\newaliascnt{rem}{thm}
\newtheorem{rem}[rem]{Remark}
\aliascntresetthe{rem}

\DeclareMathOperator{\conv}{conv}
\DeclareMathOperator{\codeg}{codeg}

\DeclareMathOperator{\integ}{int}

\newcommand{\R}{{\mathds{R}}}
\newcommand{\Z}{{\mathds{Z}}}
\newcommand{\Pro}{{\mathds{P}}}
\newcommand{\C}{\mathds{C}}

\newcommand{\Fm}{{\mathcal{F}}}

\newcommand{\rleft}{\mathopen{}\mathclose\bgroup\left}
\newcommand{\rright}{\aftergroup\egroup\right}

\newcommand{\Zspan}[1]{\rleft<{#1}\rright>}

\DeclareMathOperator{\aff}{aff}

\newcommand{\set}[1]{\rleft\{ {#1} \rright\}}

\DeclareMathOperator{\codim}{codim}


\begin{document}
\selectlanguage{english}


\title{On defectivity of families of full-dimensional point configurations}

\author{Christopher Borger}
\address{
Christopher Borger, Fakult\"at f\"ur Mathematik, 
Otto-von-Guericke-Universit\"at Magdeburg, 
Universit\"atsplatz 2, 
39106 Magdeburg, Germany
}

\email{christopher.borger@ovgu.de}

\author{Benjamin Nill}
\address{
Benjamin Nill, Fakult\"at f\"ur Mathematik, 
Otto-von-Guericke-Universit\"at Magdeburg, 
Universit\"atsplatz 2, 
39106 Magdeburg, Germany
}

\email{benjamin.nill@ovgu.de}

\begin{abstract}
	The mixed discriminant of a family of point configurations 
	can be considered as a generalization of the $A$-discriminant of one Laurent polynomial to a family of Laurent polynomials. Generalizing the concept of defectivity, a family of point configurations is called defective if the mixed discriminant is trivial. Using a recent criterion by Furukawa and Ito
	we give a necessary condition for defectivity of a family in the case that all point configurations are full-dimensional. This implies the conjecture by Cattani, Cueto, Dickenstein, Di Rocco and Sturmfels that a family of $n$ full-dimensional 
	configurations in $\Z^n$ is defective
	if and only if the mixed volume of the convex hulls of its elements
	 is $1$.
\end{abstract}

\maketitle

\section{Introduction}

Let us fix some notation. Throughout the paper, a \emph{configuration} $A \subset \Z^n$ denotes a finite subset of $\Z^n$. 
We write $A_0 + A_1 \coloneqq \set{a_0 + a_1: a_0 \in A_0, a_1 \in A_1}$
for the \emph{Minkowski sum} of two configurations $A_0,A_1 \subset \Z^n$.
We denote by $e_1,\dots,e_n$ the standard basis vectors in $\Z^n$ and in this context
also set $e_0 \coloneqq 0 \in \Z^n$. Furthermore we denote by
$\Delta_k \coloneqq \set{e_0,e_1,\dots,e_k}$ the vertices of the \emph{standard 
unimodular simplex}. The \emph{dimension} of $A \subset \Z^n$ is the dimension of
its affine hull (which we denote by $\aff(A)$) as an affine subspace
of $\R^n$ and is denoted by $\dim(A)$. 
We call $A$ \emph{full-dimensional} if $\dim(A) = n$.
We say that two configurations $A \subset \Z^n$, $B \subset \Z^m$ are 
\emph{isomorphic}
and denote this by $A \cong B$
if there is an affine lattice isomorphism of the ambient lattices 
$\aff(A) \cap \Z^n \to \aff(B) \cap \Z^m$ mapping $A$ onto $B$. A lattice polytope that is isomorphic to a standard unimodular simplex is called \emph{unimodular simplex}. 
If a lattice homomorphism $\varphi \colon \Z^n \to \Z^m$ is surjective, we call
$\varphi$ a \emph{lattice projection}.
For convenience we use the notation $[m] \coloneqq \set{0,\dots,m}$.
\newline
\newline
Let us recall the definition of the mixed discriminant (see \cite{MixedDiscriminants}).
Consider a configuration $A \subset \Z^n$.
We say that $f \in \C[x,x^{-1}] = \C[x_1,x_1^{-1}\dots,x_n,x_n^{-1}]$ is
\emph{supported on} $A$ if
it is of the form
\begin{align*}
	f = \sum_{a \in A} c_a x^a,
\end{align*} 
with $c_a \in \C$ for all $a \in A$. 
We call an isolated solution $u \in (\C^*)^n$
for a system of Laurent
polynomials $f_0(x)=\dots=f_k(x)=0$ a \emph{non-degenerate multiple root}, if the gradients $\nabla f_i(u)$ are linearly dependent, 
while any $k$ of them are linearly independent.
Now consider $A_0,\dots,A_k \subset \Z^n$. 
Each polynomial $f_i$ supported on $A_i$ is of the form $f_i = \sum_{a \in A_i} c_{i,a} x^a$ 
and we define the \emph{discriminantal variety} $\Sigma_{A_0,\dots,A_k}$
 as the closure of the set of coefficients $c_{i,a}$
such that the corresponding system of the Laurent polynomials $f_i$ has a
non-degenerate multiple root. If $\Sigma_{A_0,\dots,A_k}$ is a hypersurface, one
defines the \emph{mixed discriminant} $\Delta_{A_0,\dots,A_k}$ to be the up
to sign unique irreducible integral polynomial defining it.
Otherwise, and this is the case we are going to be interested in, we
set $\Delta_{A_0,\dots,A_k}=1$ and call the set of 
configurations $A_0,\dots,A_k$ \emph{defective}.
\newline

In the specific case of a single configuration $A \subset \Z^n$ the mixed discriminant $\Delta_A$ agrees with the \emph{$A$-discriminant} as introduced in \cite{GKZ}. Let us recall the relation of defectivity of a point configuration to defectivity of projective varieties. 
Let $A = \set{a_0,\dots,a_k} \subset \Z^n$ and denote by 
$X_A \subseteq \Pro^k$ the toric variety
obtained as the closure of the image of the morphism
\begin{align*}
	\varphi_A \colon (\C^*)^n \to \Pro^k \qquad t \mapsto 
	[t^{a_0} \colon \dots \colon t^{a_k}].
\end{align*}
Then the variety $X_A^*$ projectively dual to $X_A$ is the same as the
projectivization of the variety $\Sigma_A$.
The \emph{dual defect} $\delta_{X_A}$ of $X_A$ is defined as
$\delta_{X_A} \coloneqq \codim(X_A^*)-1$, and the variety $X_A$ is called \emph{defective} if $\delta_{X_A} > 0$. In particular, $X_A$ is 
defective if and only if $A$ is a defective configuration, or equivalently, the degree of the $A$-discriminant is zero. 
The $A$-discriminant, especially its degree, has been studied intensively starting with the book \cite{GKZ}. We refer to the survey article \cite{Piene-survey} for background and references. In particular, a special focus has been on the question of defectivity when $A$ is the set of all lattice points of its convex hull (\cite{DiRoccoSmooth}, \cite{DiRoccoCasagrande}, \cite{diRocco_dickenstein_piene}, \cite{dickenstein_nill}, \cite{dickenstein_nill_vergne}). In more general situations, conditions for defectivity were given in \cite{CC}, \cite{tropicalDiscriminants},\cite{EsterovNP}, \cite{ItoCayley}. In particular, a complete characterization in terms of so-called \emph{iterated circuits} was presented by Esterov \cite{EsterovNP} and proven in \cite{esterovAffineVarieties} (see also \cite{forsgard} for a more general version). Recently, a different characterization was obtained by Furukawa and Ito \cite{Ito} phrased in terms of so-called \emph{Cayley sums} (we refer the reader to Section~2 for the definition of Cayley sums).

The study of defectivity of a family of point configurations 
has so far been addressed in \cite{MixedDiscriminants},
 \cite{Dim2mixeddef}, \cite{Esterov2019}
and, using a slightly different definition of defectivity of a family, in \cite{EsterovNP}. By the so-called Cayley trick, their defectivity can be reduced to defectivity of their Cayley sum if all point configurations are full-dimensional (see Theorem~\ref{thm:reducetocayley}). Using the recent results by Furukawa and Ito, this allows us to deduce a necessary condition for defectivity of a family. For this, let us introduce some notation. For $A \subset \Z^n$ we denote by
$\rleft<A - A \rright>$ the subgroup of $\Z^n$ generated by the set
$\set{a_1 - a_2 \colon a_1,a_2 \in A}$ and say that $A \subset \Z^n$ is spanning, if $\rleft<A - A \rright> = \Z^n$. More generally we
say that a family $A_0,\dots,A_k \subset \Z^n$ is \emph{spanning}, if
$\rleft<A_0 - A_0 \rright> + \dots + \rleft<A_k - A_k \rright> = \Z^n$.

\begin{thm}
	\label{thm:mainthm}
	Let $k \leq n$ and $A_0,\dots,A_k \subset \Z^n$ full-dimensional configurations that form a spanning family. If $A_0,\dots,A_k$ is
	defective, then the convex hull of the Minkowski sum $A_0 + \dots + A_k$
	does not have any interior lattice points, i.e., 
	\begin{align*}
		\integ(\conv(A_0 + \dots + A_k)) \cap \Z^n = \emptyset.
	\end{align*}
\end{thm}

As a consequence, we get the following result which was conjectured in \cite{MixedDiscriminants}, where it was proven in the $2$-dimensional case as well as under additional smoothness assumptions. 

\begin{cor}
	Let $A_0,\dots,A_{n-1} \subset \Z^n$ be a spanning family of full-dimensional configurations. 
	Then $A_0,\dots,A_{n-1}$ is defective if and only if it has mixed volume $1$. In this case, $A_0, \ldots, A_{n-1}$ are all translates of the vertex set of the same unimodular simplex.
	\label{maincor}
\end{cor}

\begin{proof}
Clearly, having mixed volume one implies defectivity.
By Theorem~1 in \cite{Khov78} (or Corollary 3.2 of \cite{MixedDegree}) 
the mixed volume of $\conv(A_0),\dots,\conv(A_{n-1})$
can be computed as
\begin{align*}
	1 + \sum_{\emptyset \neq I \subseteq [n-1]} (-1)^{n-|I|} 
	|\integ(\conv(\sum_{i \in I} A_i)) \cap \Z^n|.
\end{align*}
If $A_0,\dots,A_{n-1}$ is defective, \autoref{thm:mainthm} 
implies $\conv(A_0 + \dots + A_{n-1})$ and therefore
(as all $A_i$ are full-dimensional) also 
$\conv(\sum_{i \in I} A_i)$ for any $I \subseteq [n-1]$ to have no
interior lattice points. This shows that the mixed volume of 
$\conv(A_0),\dots,\conv(A_{n-1})$ is $1$.
The last statement follows from Proposition 2.7 of \cite{MixedDiscriminants} (see also \cite{EsterovGusev}).
\end{proof}

\begin{rem}
After the first version of this paper has been made available, there
has been given another proof of Corollary~\ref{maincor} by 
Esterov (Corollary 3.23 in \cite{Esterov2019}). Esterov's result
is more general in the sense that it only makes the weaker assumption
of $A_0,\dots,A_{n-1}$ forming a so-called irreducible family instead 
of all configurations being full-dimensional. However, it does not 
generalize Theorem~\ref{thm:mainthm}, as it only treats the case
of $k=n-1$. It would be interesting to investigate whether the 
assumption of full-dimensionality in Theorem~\ref{thm:mainthm}
can always be replaced by irreducibility of the family.
We call a family $A_0,\dots,A_k \subset \Z^n$ \emph{irreducible}
if no $l$ distinct members can be shifted to a common
$(l+(n-1-k))$-dimensional affine subspace for any 
$l \in \{1,\dots,k\}$.
\end{rem}

Note that for given $A_0,\dots,A_k \subset \Z^n$ one may always choose a spanning
family whose mixed discriminantal variety equals $\Sigma_{A_0,\dots,A_k}$. By applying a suitable transformation, this implies the following slightly more 
general version of \autoref{thm:mainthm}. 

\begin{cor}
	\label{cor:lambdacor}
	Let $k \leq n$ and $A_0,\dots,A_k \subset \Z^n$ full-dimensional configurations.
	Define $\Lambda \coloneqq \rleft < A_0 - A_0 \rright > + \dots +
	\rleft <A_k-A_k\rright >$ the lattice spanned by these
	configurations. If $A_0,\dots,A_k$ is defective then 
	\begin{align*}
		\integ((A_0-a_0) + \dots + (A_k-a_k)) \cap \Lambda
	 	= \emptyset,
	\end{align*}		
	for all choices $a_0,\dots,a_k$ such that $a_i \in A_i$
	for all $i \in [k]$.
\end{cor}

\begin{rem}
	The statement of \autoref{thm:mainthm} is in general not true if we 
	do not pose sufficient restrictions on the dimensions of the 
	configurations.
	A counterexample is provided by choosing
	$A_0,A_1 \subset \Z^2$ as 
	
	\begin{align*}
		A_0 = 
		\begin{pmatrix}
			0 & 1 & 2 \\
			0 & 0 & 0 
		\end{pmatrix} \text{ and }
		A_1 =
		\begin{pmatrix}
			0 & 0 & 0 \\
			0 & 1 & 2
		\end{pmatrix}.						
	\end{align*}
	It is straightforward to verify that the corresponding system 
	\begin{align*}
		f_0 = c_{0,00} + c_{0,10}x_1 + c_{0,20}x_1^2,\text{ } 
		f_1 = c_{1,00} + c_{1,01}x_2 + c_{1,02}x_2^2, 
	\end{align*}
	does not have a non-degenerate multiple root for any choice of 
	coefficients. Therefore the variety $\Sigma_{A_0,A_1}$ is empty,
	in particular $A_0,A_1$ is a defective family, while 
	$\conv(A_0 + A_1)$ contains $(1,1)$ as an interior lattice point.
\end{rem}

\begin{rem}
	Note that the criterion for defectivity given in \autoref{thm:mainthm} is not
	sufficient. An easy class of counterexamples is given for $k=0$ by 
	$A_0 \coloneqq \conv(n \Delta_n) \cap \Z^n$ for $n > 1$. Clearly 
	$\conv(A_0)$ does not
	have any interior lattice points but cannot be defective since its lattice width 
	is $n > 1$.	
\end{rem}

\subsection*{Organization of the paper} In Section~2 we introduce Cayley sums and recall some basic results. Section~3 contains the proof of \autoref{thm:mainthm}. 

\subsection*{Acknowledgments} We thank Alicia Dickenstein and Sandra Di Rocco for fruitful discussions and for sharing their notes
about Theorem~\ref{thm:reducetocayley}.
We furthermore thank Alexander Esterov for his interest and helpful comments.
This work was funded by the Deutsche Forschungsgemeinschaft (DFG, German Research Foundation) - 314838170, GRK 2297 MathCoRe. The second author is an affiliated researcher with Stockholm University and partially supported by the Vetenskapsr{\aa}det grant~NT:2014-3991. 

\section{Basics of Cayley Sums}

As Cayley sums are going to play a crucial role in our proof, let us
recall some basic facts.

\begin{defn}
Let $A_0, \dotsc , A_k \subset \Z^n$ be configurations. We define the \emph{Cayley sum} $A_0 * \dots * A_k$ as 
$$
A_0 * \dots * A_k \coloneqq (A_0 \times \{e_0\}) \cup (A_1 \times \{e_1\}) \cup \dots \cup (A_k \times \{e_k\}) \subset \Z^{n+k}. 
$$
We call a Cayley sum $A_0 * \dots * A_k$ \emph{proper} if all $A_i$ are
non-empty.
In this case one has $\dim(A_0 * \dots * A_k) = \dim(A_0 + \dots + A_k) + k$.
\end{defn}

Let $F \subseteq A$ be a subconfiguration of a configuration $A \subset \Z^n$.
We denote by $F^c = \set{x \in A \colon x \notin F}$ the \emph{complement of $F$ in $A$}. Furthermore, we
 call $F$ a \emph{face}
of $A$ if it is the intersection of a face of the lattice polytope $\conv(A)$ with $A$ and denote
by $\Fm(A)$ the set of all faces of $A$. We call a face $F \in \Fm(A)$
\emph{proper} if $F \neq A$.

\begin{defn}
	Let $A \subset \Z^n$ and $F_0,\dots,F_k \in \Fm(A)$ faces that cover $A$.
	We say that $F_0,\dots,F_k$ form a \emph{Cayley decomposition} of $A$
	if there exists a lattice projection $\pi \colon \Z^n \to \Z^k$ such that
	$\pi(F_i) \subseteq \set{e_i}$ for all $i \in [k]$. 
\end{defn}

\begin{rem}
\label{tilde_notation}
Clearly, a Cayley sum $A_0 * \dots * A_k$ has a Cayley 		decomposition into the faces $(A_0 \times \set{e_0}),\dots,(A_k \times \set{e_k})$ and we denote them
 by $\tilde{A_i} \coloneqq A_i \times \set{e_i}.$ 
\end{rem}

\begin{prop}
	\label{prop:basicCayleyEq}
	Let $A \subset \Z^n$ be a configuration. Then the following are equivalent.
	\begin{enumerate}
		\item \label{itm:n1} There exists a Cayley decomposition of $A$ into non-empty faces
		$F_0,\dots,F_k \in \Fm(A)$,
		\item \label{itm:n2} there exists a lattice projection $\pi \colon \Z^n \to \Z^k$
		with $\pi(A) = \Delta_k$,  
		\item \label{itm:n3} there exist configurations $A_0,\dots,A_k \subset \Z^{n-k}$ 
		such that $A \cong A_0 * \dots * A_k$.
	\end{enumerate}
\end{prop}

The proof is left to the reader (c.f. \cite[Prop.2.3]{NillBatyrev}).

\begin{rem}
	\label{rem:faces}
	Let $A \subset \Z^n$ be a configuration, $F_0,\dots,F_k \in \Fm(A)$ 
	a Cayley decomposition of $A$ and $F \in \Fm(A)$ an arbitrary face. 
	Then we have a Cayley decomposition
	\begin{align*}
		F \cong (F_0 \cap F) * \dots * (F_k \cap F).
	\end{align*} 
	In particular, any face of a Cayley sum $A_0 * \dots * A_k$ is 
	isomorphic to a Cayley sum of (maybe empty) faces of each of the $A_i$.
\end{rem}

%
	

\begin{defn}
	Let $A_0,\dots,A_k \subset \Z^n$ be configurations. We say that the Cayley sum
	$A_0 * \dots * A_k$ is \emph{of join type} if the homomorphism
	$\rleft<A_0 - A_0 \rright> \oplus \dots \oplus \rleft< A_k - A_k \rright> 
	\to \Zspan{A_0 - A_0} + \dots + \Zspan{A_k - A_k} \subset \Z^n$ given by 
	$\rleft(a_0,\dots,a_k\rright) \mapsto a_0 + \dots + a_k$
	is injective.
\end{defn}

\begin{rem}
	\label{rem:dimofcayleyjoin}
	As $\dim(\Zspan{A_0 - A_0} \oplus \dots
	\oplus \Zspan{A_k - A_k}) = \dim(A_0) + \dots + \dim(A_k)$ and 
	$\dim(\Zspan{A_0 - A_0} + \dots 
	+ \Zspan{A_k - A_k}) = \dim(A_0 + \dots + A_k)$ a Cayley sum 
	$A_0 * \dots * A_k$ is of join type if and only if 
	$\dim(A_0) + \dots + \dim(A_k) = \dim(A_0 + \dots + A_k)$.
	In particular, the dimension of a proper Cayley sum $A_0 * \dots * A_k$ of 
	join type equals $\dim(A_0) + \dots + \dim(A_k) + k$ which is the
	maximal Cayley
	dimension for given dimensions of the summands $A_0,\dots,A_k$.
\end{rem}




\section{Proof of Main Theorem}

The following result has been presented by Di Rocco in a talk in June 2016 at the Fields Institute for Research in Mathematical Sciences and is soon to appear in an announced paper by Di Rocco, Dickenstein and
 Morrison
\cite{cayley_mixed} (see also \cite{MixedDiscriminants} for the 
special case where $k=n-1$).

\begin{thm}
\label{thm:reducetocayley}
If a family of configurations $A_0,\dots,A_k \subset \Z^n$ is defective,
then the Cayley sum 
$A_0 * \dots * A_k \subset \Z^{n+k}$ is defective.
\end{thm}

This identification allows us to apply the following
characterization of 
defective configurations by Furukawa and Ito \cite{Ito}
as the main tool in proving our statement about defectivity of a family
of configurations.

\begin{thm}[Furukawa, Ito]
	\label{thm:ito}
	Let $A \subset \Z^n$ be a spanning configuration.
	Then $A$ is defective if and only if there exist natural numbers $c < r$ and
	a lattice projection $\pi \colon \Z^n \to \Z^{n-c}$ such that 
	$\pi(A) \cong B_0 * \dots * B_r$ where this Cayley sum
	$B_0 * \dots * B_r$ is of join type and
	$B_i \neq \emptyset$ for all $i \in [r]$. 
\end{thm}

It is a straightforward computation to show that $A_0,\dots,A_k \subset \Z^n$
form a spanning family if and only if their Cayley sum
$A_0 * \dots * A_k \subset \Z^{n+k}$ is spanning.

The following technical lemma is crucial for the proof of the main theorem. 

\begin{lem}
	\label{lem:bdim}
	Let $A_0,\dotsc,A_k \subset \Z^n$ be full-dimensional configurations
	and $B_0,\dots,B_r \subset \Z^{n+k-r}$ non-empty configurations, such
	that 
	\begin{align*}
		A_0 * \dots * A_k \cong B_0 * \dots * B_r \subset \Z^{n+k}.
	\end{align*}
	
	\begin{enumerate}[label=(\alph*)]
		\item 
		\label{itm:firstineq}
		One has $dim(B_i) \geq \min(k,n)$ for all $i \in [r]$.
		
		\item 
		\label{itm:mainineq}
		If furthermore $\dim(B_i) < n$ for all $i \in [r]$,
		 also the following inequality holds:
		\begin{align*}
			\dim(B_0) + \dots + \dim(B_r) \geq  n-r + (r+1)k .
		\end{align*}
	\end{enumerate}   
	
\end{lem}

\begin{proof}
For $k = 0$ or $r = 0$ one can directly verify that both statements hold.
So we may assume $k,r \geq 1$ and observe that in this case each of the 
$\tilde{B}_i \subseteq B_0 * \dots * B_r$ (see \autoref{tilde_notation})
  is isomorphic to a proper face
 $B_i' \subseteq A_0 * \dots * A_k$ 
and $B_0', \dots , B_r'$ form a Cayley decomposition of $A_0 * \dots * A_k$ 
(since the $\tilde{B}_i$ form a Cayley decomposition of $B_0 * \dots * B_r$).
The complement $(B_i')^c$ of each of the $B_i'$ is again a proper face of 
$A_0 * \dots  * A_k$ (since this is true for the complement of $\tilde{B}_i$).
Let now $i \in [r]$ be arbitrary and assume $\dim(B_i)<n$ (otherwise
\ref{itm:firstineq} is trivial).
Then $B_i'$  cannot contain $\tilde{A}_j$ for any $j \in [k]$ and 
 $(B_i')^c$ has non-empty intersection with each of the $\tilde{A}_j$.
 Therefore
	by \autoref{rem:faces}
in particular $\dim (B_i')^c \geq \dim ((B_i')^c \cap \tilde{A}_j) + k$ for all $j \in [k]$. 
If now
$(B_i')^c$ contained one of the $\tilde{A}_j$, this inequality would 
imply $\dim(B_i')^c \geq n+k$ in contradiction to $(B_i')^c$ being a 
proper face of $A_0 * \dots * A_k$.
So also $B_i'$ has non-empty intersection with all of 
the  $\tilde{A}_j$ and 
by \autoref{rem:faces} we have
\begin{align*}
B_i' \cong (\tilde{A}_0 \cap B_i') * \dots * (\tilde{A}_k \cap B_i'),
\end{align*}
which implies
\begin{align}
	\label{ineq:dim}
	\dim(\tilde{A}_j \cap B_i') \leq \dim(B_i') - k,
\end{align}
for all $j \in [k]$ and all $i \in [r]$ with $\dim(B_i) < n$. This in 
particular implies
$\dim(B_i)=\dim(B_i') \geq k \geq \min(k,n)$. 
Moreover, since the $B_i'$ also form a Cayley decomposition 
of $A_0 * \dots * A_k$, we obtain
$$
\tilde{A}_j \cong (\tilde{A}_j \cap B_0') * \dots * (\tilde{A}_j \cap B_r')
$$
and therefore assuming $\dim(B_i) < n$ for all $i \in [r]$ applying 
\eqref{ineq:dim} yields
\begin{align*}
	n = \dim(\tilde{A}_j) &\leq r + \dim(\tilde{A}_j \cap B_0') + 
	\dots + \dim(\tilde{A}_j \cap B_r') \\
	&\leq r + \dim(B_0') - k + \dots + \dim(B_r') - k.
\end{align*}






\end{proof}

Note that the result above remains true in the more general setting
of point configurations in $\R^n$ and the notion of 
isomorphy induced by affine bijections. 

Let us recall that the \emph{codegree} $\codeg(P)$ of a lattice polytope 
$P \subset \R^n$ is the smallest natural number $c \geq 1$ such
that $\integ(cP)\cap\Z^n \neq \emptyset$ (see e.g. \cite{dickenstein_nill}). 


\begin{proof}[Proof of \autoref{thm:mainthm}]
	As remarked above, \autoref{thm:reducetocayley} implies that
	$A_0 * \dots * A_k \subset \Z^{n+k}$ is a spanning defective configuration. 
	By \autoref{thm:ito} 
	 there exist $c < r$ and a lattice projection $\pi \colon
	 \Z^{n+k} \to \Z^{n+k-c}$ such that $\pi(A_0 * \dots * A_k)$ has a 
	 Cayley decomposition of join type into non-empty faces
	$F_0,\dots,F_r \in \Fm(\pi(A_0 * \dots * A_k))$ .
	Let us assume that $\conv(A_0 + \dots + A_k)$ has interior lattice
	 points. 
	By the well-known connection between Cayley sums and weighted Minkowski sums (see e.g. \cite{Cayley})
	this is equivalent to $(k+1) \cdot \conv(A_0 * \dots * A_k)$ having an 
	interior point in $\Z^{n+k}$,
	which implies $\codeg(\conv(A_0 * \dots * A_k)) \leq k+1$.
	By \autoref{prop:basicCayleyEq} we have a projection 
	$\pi_r \colon \Z^{n+k-c} \to \Z^r$
	that maps $\pi(A_0 * \dots * A_k)$ surjectively onto $\Delta_r$. Since under
	lattice projections the codegree of a lattice polytope cannot increase 
	we get inequalities
	
	\begin{align*}
		k+1 \geq \codeg(A_0 * \dots * A_k) \geq \codeg(F_0 * \dots * F_r) \geq
		\codeg(\Delta_r) = r + 1, \text{ hence}
	\end{align*} 

	\begin{align}
		\label{ineq:k>=r}
		k \geq r.
	\end{align}
	 
	We observe that the lifts
	\begin{align*}
		\hat{F}_i \coloneqq \pi^{-1}(F_i) \cap (A_0 * \dots * A_k)
	\end{align*}
	define a Cayley decomposition (in general not of join type) of 
	$A_0 * \dots * A_k$.
	As $\pi$ is a projection of codimension $c$, we see 
	
	\begin{align}
		\label{ineq:liftdim}
		\dim(\hat{F}_i) \leq \dim(F_i) + c,	
	\end{align}
	
	for all $i \in [r]$. Combining this with 
	the fact
	that the $F_i$ form a Cayley decomposition of join type and using
	\autoref{rem:dimofcayleyjoin} one obtains
	
	\begin{align*}
		\dim(\hat{F}_0) + \dots + \dim(\hat{F}_r) 
		&\leq \dim(F_0) + \dots + \dim(F_r) + c(r+1)\\
		&= \dim(F_0 + \dots + F_r) + c(r+1)\\
		&= n + k - c - r + c(r+1)\\
		&= n + k + r(c-1).
	\end{align*}
	Let us assume $\dim(\hat{F}_j) \geq n$ for some $j \in [r]$ .
	Therefore $\dim(F_j) \geq n-c$. 
	Without loss of generality let $j=0$.
	As the $F_i$ form a Cayley decomposition of join type of the 
	$(n+k-c)$-dimensional configuration $\pi(A_0 * \dots * A_k)$ we have
	the following inequality for the remaining summands:
	\begin{align*}
		\dim(F_1) + \dots + \dim(F_r) 
		&= \dim(F_0 + \dots + F_r) - \dim(F_0)\\
		&= n + k - c - r - \dim(F_0)\\
		&\leq n + k - c - r - (n-c)\\
		&= k -r. 
	\end{align*}
	However, on the other hand \autoref{lem:bdim}~\ref{itm:firstineq} implies $\dim(\hat{F}_i) \geq k$
	for all
	$i \in [r]$ (since we assumed $k \leq n$). So by \eqref{ineq:liftdim} 
	we have
	$\dim(F_i) \geq k - c$ which yields another inequality for the remaining
	summands:
	\begin{align*}
		\dim(F_1) + \dots + \dim(F_r) \geq r(k-c).
	\end{align*}
	These inequalities contradict each other since $r(k-c)>k-r$, which can 
	be seen
	by observing that $r$ is strictly positive and $c$ is strictly smaller 
	than $r$.
	
	Therefore $\dim(\hat{F}_j)<n$ for all $j \in [r]$.		 
	So we may apply part \ref{itm:mainineq} of \autoref{lem:bdim} and obtain
	$n-r+(r+1)k \leq \dim(\hat{F}_0) + \dots + \dim(\hat{F}_r)$. Hence,
	
	\begin{align*}
		n-r+(r+1)k \leq n + k + r(c-1),
	\end{align*}
	
	which is (since $r$ is strictly positive) equivalent to $k \leq c < r$,
	a contradiction.	
\end{proof}

\bibliographystyle{amsalpha}
\bibliography{MixedDefectivityBib}

\end{document}